\documentclass[smallcondensed,envcountsame,envcountsect]{svjour3}

\usepackage{amssymb,amsmath}
\usepackage{mathptmx}

\smartqed

\newcommand{\F}{\mathbb{F}}

\newcommand{\N}{\mathbb{N}}
\newcommand{\Ss}{\mathbb{S}}

\newcommand{\bq}{\mathbf{q}}

\DeclareMathOperator{\ord}{ord}
\DeclareMathOperator{\PCN}{PCN}
\DeclareMathOperator{\CN}{CN}
\DeclareMathOperator{\lcm}{lcm}

\title{Further results on the Morgan-Mullen conjecture}
\author{Theodoulos Garefalakis \and Giorgos Kapetanakis}
\institute{Theodoulos Garefalakis \at Department of Mathematics and Applied Mathematics, University of Crete, Voutes Campus, 70013 Heraklion, Greece \\ \email{tgaref@uoc.gr} \\
Giorgos Kapetanakis \at Sabanci University, FENS, Orhanli-Tuzla, 34956 Istanbul, Turkey \\ \email{gnkapet@gmail.com}
}

\begin{document}
\maketitle
\begin{abstract}
  Let $\F_q$ be the finite field of characteristic $p$ with $q$ elements and $\F_{q^n}$ its extension of degree $n$.
  The conjecture of Morgan and Mullen asserts the existence of primitive and completely normal elements (PCN elements) for the extension
  $\F_{q^n}/\F_q$ for any $q$ and $n$. It is known that the conjecture holds for $n \leq q$. In this work we prove the conjecture
  for a larger range of exponents. In particular, we give sharper bounds for the number of completely normal elements and use them to prove asymptotic and effective existence results for $q\leq n\leq O(q^\epsilon)$, where $\epsilon=2$ for the asymptotic results and $\epsilon=1.25$ for the effective ones. For $n$ even we need to assume that $q-1\nmid n$.
\end{abstract}

\section{Introduction}\label{sec:intro}

Let $\F_q$ be the finite field of cardinality $q$ and $\F_{q^n}$ its extension of degree $n$, where $q$ is a prime power and $n$ is a positive integer.
A generator of the multiplicative group $\F_{q^n}^*$ is called \emph{primitive}. Besides their theoretical interest, primitive
elements of finite fields are widely used in various applications, including
cryptographic schemes, such as the Diffie-Hellman key exchange
\cite{diffiehellman76}.

An \emph{$\F_q$-normal basis} of $\F_{q^n}$ is an $\F_q$-basis of $\F_{q^n}$ of the form $\{ x, x^q, \ldots , x^{q^{n-1}} \}$ and the element $x\in\F_{q^n}$ is called \emph{normal over $\F_q$}. These bases bear computational advantages for finite field
arithmetic, so they have numerous applications, mostly in coding theory and cryptography. For further information we refer to \cite{gao93} and the references therein.

It is well-known that primitive and normal elements exist for every $q$ and $n$, see Chapter~2 of \cite{lidlniederreiter97}. The existence of
elements that are simultaneously primitive and normal is also well-known.
\begin{theorem}[Primitive normal basis theorem]\label{thm:pnbt}
Let $q$ be a prime power and $n$ a positive integer. There exists some $x \in
\F_{q^n}$ that is simultaneously primitive and normal over $\F_q$.
\end{theorem}
Lenstra and Schoof \cite{lenstraschoof87} were the first to
prove Theorem~\ref{thm:pnbt}. Subsequently, Cohen
and Huczynska \cite{cohenhuczynska03} provided a computer-free proof with the
help of sieving techniques. Several generalizations of this have also been investigated
\cite{cohenhachenberger99,cohenhuczynska10,hsunan11,kapetanakis13,kapetanakis14}.

An element of $\F_{q^n}$ that is simultaneously normal over $\F_{q^l}$ for all $l\mid n$ is called \emph{completely normal over $\F_q$}. The existence of such elements for any $q$ and $n$ is well-known \cite{blessenohljohnsen86}. Morgan and Mullen \cite{morganmullen96} conjectured that for any $q$ and $n$, there exists a primitive completely normal element of $\F_{q^n}$ over $\F_q$.
\begin{conjecture}[Morgan-Mullen]\label{conj:mm}
Let $q$ be a prime power and $n$ a positive integer. There exists some $x \in
\F_{q^n}$ that is simultaneously primitive and completely normal over $\F_q$.
\end{conjecture}
In order to support their claim, Morgan and Mullen provide examples for such elements for all pairs $(q,n)$ with $q\leq 97$ and $q^n<10^{50}$, see \cite{morganmullen96}. This conjecture is yet to be completely resolved. Partial results, covering certain types of extensions have been given, see \cite{hachenberger13} and the references therein. Recently, Hachenberger \cite{hachenberger16}, using elementary methods, proved the validity of Conjecture~\ref{conj:mm} for $q\geq n^3$ and $n\geq 37$.
In \cite{garefalakiskapetanakis18}, the range was improved to $n\leq q$.

In this work we extend the range for which Conjecture~\ref{conj:mm} holds. In particular, we prove the following theorems:
\begin{theorem}\label{thm:our_asymptotic}
  There exists $c\in\N$ such that for every prime power $q\geq c$
  and every $n\in\N$ satisfying
  \begin{enumerate}
  \item $n$ odd, and $q\leq n\leq q^2$, or
    \item $n$ even, $q-1\nmid n$ and $q\leq n\leq 0.43\cdot q^2$,
  \end{enumerate}
  there exists a primitive and completely normal element for the
  extension $\F_{q^n}/\F_q$
\end{theorem}
\begin{theorem}\label{thm:our_effective}
Let $q$ be a prime power and $n$ an integer. There exists a primitive element of $\F_{q^n}$ that is completely normal over $\F_q$ in the following cases:
\begin{enumerate}
  \item $n$ is odd and $n<q^{4/3}$ and
  \item $n$ is even, $q-1\nmid n$ and $n<q^{5/4}$.
\end{enumerate}
\end{theorem}
In Section~\ref{sec:prelim}, we prove our main technical tool, Theorem~\ref{thm:our2}. In Section~\ref{sec:cn}, we prove some new bounds for the number of completely normal elements. In Section~\ref{sec:proof1} we combine Theorem~\ref{thm:our2} and the bounds of Section~\ref{sec:cn} to establish Theorem~\ref{thm:our_asymptotic}. Next, in Section~\ref{sec:proof2}, we combine Theorem~\ref{thm:our2} and the bounds of Section~\ref{sec:cn} to establish Theorem~\ref{thm:our_effective}, for all but a small number of possible exceptions, that are dealt with, either by employing the Cohen-Huczynska \cite{cohenhuczynska03,cohenhuczynska10} sieving techniques, or by relying on the Morgan-Mullen \cite{morganmullen96} examples. We conclude this work with some remarks about possible further improvements in Section~\ref{sec:conclusions}.
\section{Preliminaries}\label{sec:prelim}
The notion of primitivity can be generalized as follows. We call $x\in\F_{q^n}$ \emph{$r$-free}, where $r \mid q^n-1$, if $x=y^d$ for some $d\mid r$ and $y\in\F_{q^n}$ implies $d=1$. Clearly, the primitive elements are exactly the $q'$-primitive elements, where $q'$ is the square-free part of $q^n-1$. In addition, notice that $r$-freeness depends solely on the prime divisors of $r$, that is one may freely interchange between $r$ and its square-free part.

By using Vinogradov's formula for generators of cyclic modules over Euclidean domains, it can be shown that the characteristic function for $r$-free elements of $\F_{q^n}$, where $r\mid q'$, is
\[
\omega_r(x) := \theta(r) \sum_{\chi\in\widehat{\F_{q^n}^*},\ \ord(\chi) \mid r}\frac{\mu(\ord(\chi))}{\phi(\ord(\chi))} \chi(x),
\]
where $\theta(r):=\phi(r)/r$, $\mu$ is the M\"obius function, $\phi$ is the Euler function and the \emph{order} of the multiplicative character $\chi$, denoted as $\ord(\chi)$, is defined as its multiplicative order in $\widehat{\F_{q^n}^*}$. Also, for the sake of simplicity, we denote $\omega := \omega_{q'}$, thus $\omega$ is the characteristic function for primitive elements.

Similarly, the characteristic function for elements of $\F_{q^n}$ that are normal over $\F_{q^l}$ is
\[
\varOmega_l(x) := \theta_l(X^{n/l}-1) 
    \sum_{\psi\in\widehat{\F_{q^n}}} \frac{\mu_l(\ord_l(\psi))}{\phi_l(\ord_l(\psi))} \psi(x) ,
\]
where $\theta_l(X^{n/l}-1):= \phi_l(F_l')/q^{l\cdot\deg(F_l')}$, $F_l'$ is the square-free part of $X^{n/l}-1\in\F_{q^l}[X]$, $\mu_l$ and $\phi_l$ are the M\"obius and Euler functions in $\F_{q^l}[X]$ respectively and the \emph{order} of an additive character $\psi$ of $\F_{q^n}$ over $\F_{q^l}$, denoted as $\ord_l(\psi)$, is defined as the lowest degree monic polynomial $G=\sum_{i=0}^m G_iX^i \in\F_{q^l}[X]$, such that $\psi \left( \sum_{i=0}^m G_i x^{q^i} \right) = 1$ for all $x\in\F_{q^n}$. It is straightforward to check that $\ord_l(\psi) \mid X^{n/l}-1$ in $\F_{q^l}[X]$.

Let $\CN_q^r(n)$ be the number of $r$-free completely normal elements of $\F_{q^n}$ over $\F_q$ and $\PCN_q(n)$ be the number of primitive completely normal elements of $\F_{q^n}$ over $\F_q$, so that $\PCN_q(n)=\CN_q^{q'}(n)$. Further, let $\CN_q(n)$ be the number of completely normal elements of $\F_{q^n}$ over $\F_q$. Assume that $\{1=l_1<\ldots <l_k<n\}$ is the set of proper divisors of $n$. Since all $x\in\F_{q^n}^*$ are normal over $\F_{q^n}$, it follows that an element of $\F_{q^n}$ is completely normal over $\F_q$ if and only if it is normal over $\F_{q^{l_i}}$ for all $i=1,\ldots ,k$. To simplify our notation, we denote $\bq=(X^{n/l_1}-1,\ldots,X^{n/l_k}-1)$ and $\theta(\bq)=\prod_{i=1}^k \theta_{l_i}(X^{n/l_i}-1)$.
We compute
\begin{eqnarray*}
\CN_q(n) &=& \sum_{x\in\F_{q^n}} \left(\varOmega_{l_1}(x) \cdots \varOmega_{l_{k}}(x) \right) \\
&=& \theta(\bq) \sum_{(\psi_1,\ldots,\psi_k)}
\prod_{i=1}^k\frac{\mu_{l_i}(\ord_{l_i}(\psi_i))}{\phi_{l_i}(\ord_{l_i}(\psi_i))}
\sum_{x\in\F_{q^n}} \psi_1\cdots\psi_k(x) ,
\end{eqnarray*}
where the sums extends over all $k$-tuples of additive characters. Noting that 
\[
\sum_{x\in\F_{q^n}} \psi_1\cdots\psi_k(x) =0,\ \ \ \mbox{ for }\ \psi_1\cdots\psi_k\neq \psi_0,
\]
we obtain
\[
  \CN_q(n) = q^n\ \theta(\bq) \sum_{\substack{(\psi_1,\ldots,\psi_k)\\ \psi_1\cdots\psi_k=\psi_0}}
\prod_{i=1}^k\frac{\mu_{l_i}(\ord_{l_i}(\psi_i))}{\phi_{l_i}(\ord_{l_i}(\psi_i))}.
\]
The following theorem is a direct generalization of \cite[Theorem~3.1]{garefalakiskapetanakis18} and it is the main technical result from which all the sufficient conditions in the proofs of Theorems~\ref{thm:our_asymptotic} and \ref{thm:our_effective} are derived.
\begin{theorem}\label{thm:our2}
Let $q$ be a prime power, $n\in\N$ and $r$ a square-free divisor of $q^n-1$, then
\[
|\CN_q^r(n) - \theta(r)\CN_q(n)| \leq q^{n/2} W(r)W_{l_1}(F_{l_1}') \cdots W_{l_k}(F_{l_k}')\theta(r)\theta(\bq) ,
\]
where $W(r)$ is the number of positive divisors of $r$ and $W_{l_i}(F_{l_i}')$ is the number of monic divisors of $F_{l_i}'$ in $\F_{q^{l_i}}[X]$.
\end{theorem}
\begin{proof}
Using the characteristic functions, as presented earlier, we deduce that
\begin{align*}
\CN_q^r(n) & = \sum_{x\in\F_{q^n}} \left(\omega_r (x) \varOmega_{l_1}(x) \cdots \varOmega_{l_{k}}(x) \right)\\
&= \theta(r)\theta(\bq)\sum_{\chi}\sum_{(\psi_1,\ldots,\psi_k)}
 \frac{\mu(\chi)}{\phi(\chi)}\prod_{i=1}^k\frac{\mu_{l_i}(\ord_{l_i}(\psi_i))}{\phi_{l_i}(\ord_{l_i}(\psi_i))}
 \sum_{x\in\F_{q^n}} \psi_1\cdots\psi_k(x) \chi(x) \\
 &= \theta(r)\theta(\bq) (S_1+S_{2,r}),
\end{align*}
where the term $S_1$ is the part of the above sum that corresponds to $\chi=\chi_0$, the trivial character. It follows that
\[
S_1 = \sum_{(\psi_1,\ldots,\psi_k)}
 \prod_{i=1}^k\frac{\mu_{l_i}(\ord_{l_i}(\psi_i))}{\phi_{l_i}(\ord_{l_i}(\psi_i))}
 \sum_{x\in\F_{q^n}} \psi_1\cdots\psi_k(x)
 = \frac{\CN_q(n)}{\theta(\bq)} .
\]
Also, $S_{2,r}$ is the part that corresponds to $\chi\neq \chi_0$,
\[
 S_{2,r} = \sum_{\chi\neq\chi_0}\sum_{(\psi_1,\ldots,\psi_k)}
 \frac{\mu(\chi)}{\phi(\chi)}\prod_{i=1}^k\frac{\mu_{l_i}(\ord_{l_i}(\psi_i))}{\phi_{l_i}(\ord_{l_i}(\psi_i))}
 \sum_{x\in\F_{q^n}} \psi_1\cdots\psi_k(x) \chi(x).
 \]
 In the last sum, note that the summations runs on multiplicative characters $\chi$ of order dividing $r$ and may be restricted to additive characters of order dividing the square-free part
 of $X^{n/l_i}-1$, which we denoted by $F_{l_i}'$.
 For the last sum we have
 \begin{align*}
  |S_{2,r}| 
  &\leq \sum_{\chi\neq\chi_0}\sum_{(\psi_1,\ldots,\psi_k)}
 \frac{1}{\phi(\ord(\chi))}\prod_{i=1}^k\frac{1}{\phi_{l_i}(\ord_{l_i}(\psi_i))}
 \left|\sum_{x\in\F_{q^n}} \psi_1\cdots\psi_k(x) \chi(x)  \right| \\
  &\leq q^{n/2} \sum_{\chi\neq\chi_0}\frac{1}{\phi(\ord(\chi))} \prod_{i=1}^k \sum_{\psi_i}\frac{1}{\phi_{l_i}(\ord_{l_i}(\psi_i))} \\
  &= q^{n/2} (W(r)-1) \prod_{i=1}^k W_{l_i}(F_{l_i}'),
 \end{align*}
 where we used the orthogonality relations and the well-known fact that for non-trivial $\chi$ and $\psi$ the absolute value of the Gauss sum $\sum_{x\in\F_{q^n}} \psi(x)\chi(x)$ is bounded by $q^{n/2}$. The result follows.
\qed\end{proof}
The following lemma is used to estimate $W(q')$, that appears above.
\begin{lemma}\label{lemma:w(r)}
For any $r\in\N$, $W(r) \leq c_{r,a} r^{1/a}$, where $c_{r,a}=2^s/(p_1 \cdots
p_s)^{1/a}$ and $p_1,\ldots ,p_s$ are the primes $\leq 2^a$ that divide $r$. In particular, $c_{r,4}<4.9$, $c_{r,12}<1.06\cdot 10^{24}$
 for all $r\in\N$.
\end{lemma}
\begin{proof}
It is clear that it suffices to prove the above for $r$ square-free. Assume that $r=p_1\cdots p_s q_1\cdots q_t$, where $p_1,\ldots, p_s,q_1, \ldots ,q_t$ are distinct primes and $p_i\leq 2^a$ and $q_j>2^a$. We have that
\[ W(r) = 2^{s+t} = 2^s \cdot \underbrace{2\cdots 2}_{t \text{ times}} = 2^s (\underbrace{2^a \cdots 2^a}_{t \text{ times}})^{1/a} \leq 2^s (q_1\ldots q_t)^{1/a} = c_{r,a} r^{1/a} . \]
The bounds for $c_{r,a}$ can be easily computed.
\qed\end{proof}
\section{Completely normal elements}\label{sec:cn}
In this section, we prove a new lower bound for $\CN_q(n)$.
Let $p$ be the characteristic of $\F_q$ and $n=p^{\ell}m$, with $(m,p)=1$.
The number of elements of $\F_{q^n}$ that are {\it not} completely normal over $\F_{q}$ is at most
$\sum_{d|n}(q^n - \phi_d(X^{n/d}-1))$. Our starting point is the following bound.
\[
  \CN_q(n)\geq q^n \left(1-\sum_{d|n}\left(1-\frac{\phi_d(X^{n/d}-1)}{q^n}\right)\right).
\]
Expressing the divisors of $n$ as $p^jd$, $0\leq j\leq \ell$, $d\, |\, m$, we have
\begin{equation}
  \label{eq:b0}
  \CN_q(p^{\ell}m) \geq q^{p^{\ell}m} \left(1 - \sum_{j=0}^{\ell}\sum_{d|m}
  \left(1 - \frac{\phi_{p^jd}(X^{p^{\ell-j}m/d}-1)}{q^{p^{\ell}m}} \right)\right).
\end{equation}
We denote $\nu_{p^jd}(k)=\ord_k(q^{p^jd})$ and to simplify notation, we let $\nu(k)=\nu_1(k)$. Then
\[
  \phi_{p^jd}\left((X^{m/d}-1)^{p^{\ell-j}} \right)
  = q^{p^{\ell}m}\prod_{k|(m/d)} \left(1 - \frac{1}{q^{p^jd\nu_{p^jd}(k)}}\right)^{\frac{\phi(k)}{\nu_{p^jd}(k)}}.
\]
For $0\leq j\leq \ell$ and $d\mid m$ we have
\begin{eqnarray*}
  \phi_{p^jd}\left((X^{m/d}-1)^{p^{\ell-j}} \right)
  &\geq& q^{p^{\ell}m}\prod_{k|(m/d)} \left(1 - \frac{1}{q^{p^jd}}\right)^{\phi(k)} \\
  &\geq& q^{p^{\ell}m}\left(1 - \frac{1}{q^{p^jd}}\right)^{\frac{m}{d}} \\
  &\geq& q^{p^{\ell}m}\left(1 - \frac{m}{dq^{p^jd}}\right).
\end{eqnarray*}
Therefore,
\begin{equation} \label{eq:b1}
  1-\frac{\phi_{p^jd}\left((X^{m/d}-1)^{p^{\ell-j}} \right)}{q^{p^{\ell}m}}
  \leq  \frac{m}{dq^{p^jd}}, \ \ \mbox{ for } 0\leq j\leq \ell,\ d\, |\, m.
\end{equation}
This bound is sufficient for all pairs $(j,d)$, except for $j=0$ and $d=1$,
which we consider separately.
\[
  \phi_{1}\left((X^{m}-1)^{p^{\ell}} \right)
  = q^{p^{\ell}m}\prod_{k|m} \left(1 - \frac{1}{q^{\nu(k)}}\right)^{\frac{\phi(k)}{\nu(k)}}.
\]
Let $g=(m,q-1)$. Then $\nu(k)=1$ if and only if $q\equiv 1\pmod{k}$, which holds if and only if
$k\, |\, g$. We have
\begin{eqnarray*}
  \frac{\phi_{1}\left((X^{m}-1)^{p^{\ell}} \right)}{q^{p^{\ell}m}}
  &=&\prod_{k|g} \left(1 - \frac{1}{q}\right)^{\phi(k)} \
      \prod_{\substack{k|m \\ k\nmid g}}\left(1 - \frac{1}{q^{\nu(k)}}\right)^{\frac{\phi(k)}{\nu(k)}}.
\end{eqnarray*}
The first product is equal to $(1-1/q)^g$, while the second is bounded as follows.
\begin{eqnarray*}
  \prod_{\substack{k|m \\ k\nmid g}}\left(1 - \frac{1}{q^{\nu(k)}}\right)^{\frac{\phi(k)}{\nu(k)}}
  &\geq& \prod_{\substack{k|m \\ k\nmid g}}\left(1 - \frac{1}{q^{2}}\right)^{\frac{\phi(k)}{2}} \\
  &=& \prod_{k|m}\left(1 - \frac{1}{q^{2}}\right)^{\frac{\phi(k)}{2}}
      \prod_{k|g}\left(1 - \frac{1}{q^{2}}\right)^{-\frac{\phi(k)}{2}} \\
  &=& \left(1-\frac{1}{q^{2}}\right)^{\frac{m-g}{2}}.
\end{eqnarray*}
Therefore, we have
\begin{eqnarray*}
  \frac{\phi_{1}\left((X^{m}-1)^{p^{\ell}} \right)}{q^{p^{\ell}m}}
  &\geq& \left( 1-\frac{1}{q} \right)^g \left( 1 - \frac{1}{q^2} \right)^{\frac{m-g}{2}}\\
  &=& \left(1-\frac{1}{q}\right)^{\frac{g}{2}} \left(1+\frac{1}{q}\right)^{-\frac{g}{2}}
         \left(1-\frac{1}{q^{2}}\right)^{\frac{m}{2}} \\
  &=& \left(1-\frac{2}{q+1}\right)^{\frac{g}{2}} \left(1-\frac{1}{q^{2}}\right)^{\frac{m}{2}} \\
  &\geq& \left(1-\frac{g}{q+1}\right) \left(1-\frac{m}{2q^{2}}\right),
\end{eqnarray*}
and we get
\begin{equation}
  \label{eq:b2}
  1-\frac{\phi_{1}\left((X^{m}-1)^{p^{\ell}} \right)}{q^{p^{\ell}m}} \leq
  \frac{g}{q+1} + \frac{m}{2q^{2}} - \frac{gm}{2q^{2}(q+1)}.
\end{equation}

Combining Eqs.~(\ref{eq:b0}), (\ref{eq:b1}) and (\ref{eq:b2}), we obtain
\begin{equation}
\label{eq:4}
  \frac{\CN_q(p^{\ell}m)}{q^{p^{\ell}m}}
  \geq 1 - \frac{g}{q+1} - \frac{m}{2q^{2}} + \frac{gm}{2q^{2}(q+1)}
  - \sum_{\substack{d|m \\ d>1}} \sum_{j=0}^{\ell} \frac{m}{d q^{p^jd}} - \sum_{j=1}^{\ell} \frac{m}{q^{p^j}}.
\end{equation}
We proceed to upper bound the sums in the last expression.
\[
  \sum_{j=1}^{\ell} \frac{m}{q^{p^j}}
  \leq \frac{m}{q^p} + \sum_{j=2}^{\infty}\frac{m}{q^{pj}} 
  = \frac{m}{q^p} + \frac{m}{q^{2p}(1-q^{-p})},
\]
and
\begin{eqnarray*}
  \sum_{j=0}^{\ell}\frac{m}{dq^{p^jd}}
  &\leq& \frac{m}{d} \left(\frac{1}{q^d} + \sum_{j=1}^{\infty}\frac{1}{q^{pdj}} \right) \\
  &\leq& \frac{m}{d} \left(\frac{1}{q^d} + \frac{1}{q^{pd}(1-q^{-pd})} \right) \\
  &\leq& \frac{m}{d} \left(\frac{1}{q^{d}} + \frac{64}{63q^{pd}}\right),
\end{eqnarray*}
where we used the fact that $pd\geq 6$, therefore $1/(1-q^{-pd})\leq 64/63$.

We now consider two cases. For $m$ odd, we have
\begin{eqnarray*}
  \sum_{\substack{d|m \\ d>1}} \sum_{j=0}^{\ell}\frac{m}{dq^{p^jd}}
  &\leq& \sum_{\substack{d|m \\ d>1}}\frac{m}{dq^{d}} + \sum_{\substack{d|m \\ d>1}} \frac{64m}{63dq^{pd}} \\
  &\leq& \frac{m}{3q^3} + \sum_{d=5}^{\infty}\frac{m}{5q^{d}} + \sum_{d=3}^{\infty}\frac{64m}{63\cdot 3 q^{pd}} \\
  &\leq& \frac{m}{3q^3} + \frac{2m}{5q^5} + \frac{m}{2q^6}.
\end{eqnarray*}
Therefore,
\begin{equation}
\label{eq:2}
  \sum_{\substack{d|m \\ d>1}} \sum_{j=0}^{\ell}\frac{m}{dq^{p^jd}} + \sum_{j=1}^{\ell} \frac{m}{q^{p^j}}
  \leq \frac{m}{q^p} + \frac{m}{3q^3}  + \frac{2m}{q^4} + \frac{2m}{5q^5} + \frac{m}{2q^6}.
\end{equation}

For $m$ even, we have $p\geq 3$ and
\[
  \sum_{\substack{d|m \\ d>1}} \sum_{j=0}^{\ell}\frac{m}{dq^{p^jd}}
  \leq \sum_{\substack{d\geq 2}}\frac{m}{dq^{d}} + \sum_{d\geq 2} \frac{64m}{63dq^{pd}}.
\]
For the sums involved, we have
\[
\sum_{d\geq 2} \frac{m}{dq^d} \leq \frac{m}{2q^2} + \frac{m}{3q^3} + \frac{m}{4q^4} \sum_{d\geq 0} \frac{1}{q^d} \leq \frac{m}{2q^2} + \frac{m}{3q^3} + \frac{m}{2q^4} ,
\]
where we used the fact that $q/(q-1) < 2$. Similarly,
\[
\sum_{d\geq 2} \frac{64m}{63dq^{pd}} \leq \frac{64m}{63\cdot 2} \sum_{d\geq 2} \frac{1}{q^{3d}} \leq \frac{32}{63} \cdot \frac{m}{q^6} \cdot \frac{q^3}{q^3-1} \leq \frac{48m}{91q^6} ,
\]
since $q^3/(q^3-1)\leq 27/26$. We conclude that
\[
\sum_{\substack{d|m \\ d>1}} \sum_{j=0}^{\ell}\frac{m}{dq^{p^jd}}
  \leq \frac{m}{2q^2} + \frac{m}{3q^3} + \frac{m}{2q^4} + \frac{48m}{91q^6}.
\]
Therefore,
\begin{equation}
\label{eq:3}
  \sum_{\substack{d|m \\ d>1}} \sum_{j=0}^{\ell}\frac{m}{dq^{p^jd}} + \sum_{j=1}^{\ell} \frac{m}{q^{p^j}}
  \leq \frac{m}{2q^2} + \frac{4m}{3q^3} + \frac{m}{2q^4} + \frac{8m}{5q^6}.
\end{equation}
We are now ready to prove the following proposition.
\begin{proposition}\label{prop:cn-bound-1}
  Let $\F_q$ be the finite field of characteristic $p$, and $n=p^{\ell}m$, with $\ell\geq 1$, $m\geq 1$,
  $(m,p)=1$.
\begin{enumerate}
\item For $m$ even
  \[
    \CN_{q}(n)
    \geq q^n \left(1-\frac{g}{q+1}+\frac{gm}{2q^2(q+1)} -
      \frac{m}{q^2} - \frac{4m}{3q^3} - \frac{m}{2q^4} - \frac{8m}{5q^6} \right).
  \]
\item For $m$ odd
  \[
    \CN_{q}(n)
    \geq q^n \left(1-\frac{g}{q+1}+\frac{gm}{2q^2(q+1)} -
     \frac{m}{2q^2} - \frac{m}{q^p} - \frac{m}{3q^3} - \frac{2m}{q^4} - \frac{2m}{5q^5} - \frac{m}{2q^6}\right).
 \]
\end{enumerate}
\end{proposition}
\begin{proof}
For $m$ even, Eq.~\eqref{eq:4} combined with the bound in Eq.~\eqref{eq:3}, we have
\begin{eqnarray*}
  \frac{\CN_q(n)}{q^n} &\geq& 1 - \frac{g}{q+1} - \frac{m}{2q^2} + \frac{gm}{2q^2(q+1)} 
  - \frac{m}{2q^2} - \frac{4m}{3q^3} - \frac{m}{2q^4} - \frac{8m}{5q^6} \\
  &\geq& 1-\frac{g}{q+1}+\frac{gm}{2q^2(q+1)} -
      \frac{m}{q^2} - \frac{4m}{3q^3} - \frac{m}{2q^4} - \frac{8m}{5q^6} .
\end{eqnarray*}
The bound for $m$ odd follows similarly from Eqs.~\eqref{eq:4} and \eqref{eq:2}.\qed
\end{proof}
For $\ell=0$, that is, $(n,p)=1$, we have slightly tighter bounds.
\begin{proposition} \label{prop:cn-bound-2}
  Let $\F_q$ be the finite field of characteristic $p$, and $n\geq 1$, $(n,p)=1$.
\begin{enumerate}
\item For $n$ even
  \[
    \CN_{q}(n)
    \geq q^n \left(1-\frac{g}{q+1}+\frac{gn}{2q^2(q+1)} - \frac{n}{q^2} - \frac{n}{2q^3}\right).
  \]
\item For $n$ odd
  \[
    \CN_{q}(n)
    \geq q^n \left(1-\frac{g}{q+1}+\frac{gn}{2q^2(q+1)} -
     \frac{n}{2q^2}-\frac{n}{3q^3}-\frac{n}{2q^4} \right).
 \]
\end{enumerate}  
\end{proposition}
\begin{proof}
  For the first bound,
  \begin{eqnarray*}
    \CN_{q}(n)
    &\geq& q^n \left(1-\frac{g}{q+1}+\frac{gn}{2q^2(q+1)}-\frac{n}{2q^2}-\sum_{\substack{d|n \\ d>1}}\frac{n}{dq^d} \right) \\
    &\geq& q^n \left(1-\frac{g}{q+1}+\frac{gn}{2q^2(q+1)}-\frac{n}{q^2}-\frac{n}{2q^3} \right)
  \end{eqnarray*}
  For the second bound,
  \begin{eqnarray*}
    \CN_{q}(n)
    &\geq& q^n \left(1-\frac{g}{q+1}+\frac{gn}{2q^2(q+1)}-\frac{n}{2q^2}-\sum_{\substack{d|n \\ d>1}}\frac{n}{dq^d} \right) \\
    &\geq& q^n \left(1-\frac{g}{q+1}+\frac{gn}{2q^2(q+1)}-\frac{n}{2q^2}-\frac{n}{3q^3}-\frac{n}{2q^4} \right).
  \end{eqnarray*}
\qed\end{proof}

\begin{corollary}\label{coro:cn-bound-1}
  Let $\F_q$ be the finite field of characteristic $p$, and $n=p^{\ell}m$, with $\ell\geq 1$, $1\leq m< 2q^2$,
  $(m,p)=1$, and $(q-1)\nmid m$.
\begin{enumerate}
\item For $m$ even, $q\geq 9$ 
  \[
    \CN_{q}(n) \geq q^n \left(\frac{1}{2} + \frac{1}{q+1} - \frac{0.96\cdot m}{q^2} \right).
  \]
\item For $m$ odd, $q\geq 8$ and $p=2$
  \[
    \CN_{q}(n) \geq q^n \left(\frac{2}{3} + \frac{2}{3(q+1)} - \frac{1.45\cdot m}{q^2}\right).
  \]
\end{enumerate}
\end{corollary}
\begin{proof}
  We proceed to prove the first bound. In the RHS expressions of the inequalities of Proposition~\ref{prop:cn-bound-1}, the quantity
  \[ - \frac{g}{q+1} + \frac{gm}{2q^2(q+1)} = -\frac{g}{q+1} \left( 1 - \frac{m}{2q^2} \right) \]
  is a decreasing function of $g$, since $m<2q^2$. Assuming that $g\leq (q-1)/2$, we have
  \begin{eqnarray*}
    & & 1-\frac{g}{q+1}+\frac{gm}{2q^2(q+1)} -
        \frac{m}{q^2} - \frac{4m}{3q^3} - \frac{m}{2q^4} - \frac{8m}{5q^6} \\
    &\geq& 1-\frac{q-1}{2(q+1)}+\frac{(q-1)m}{4q^2(q+1)} -
           \frac{m}{q^2} - \frac{4m}{3q^3} - \frac{m}{2q^4} - \frac{8m}{5q^6} \\
    &=& \frac{1}{2} + \frac{1}{q+1} -
        \frac{m}{q^2} \left(-\frac{q-1}{4(q+1)}+1+\frac{4}{3q}+\frac{1}{2q^2}+\frac{8}{5q^4}\right) \\
    &\geq& \frac{1}{2} + \frac{1}{q+1} - \frac{0.96\cdot m}{q^2}.        
  \end{eqnarray*}
  For the second bound, letting $g\leq (q-1)/3$, we have
  \begin{eqnarray*}
    & & 1-\frac{g}{q+1}+\frac{gm}{2q^2(q+1)} -
        \frac{m}{2q^2} - \frac{m}{q^p} - \frac{m}{3q^3} - \frac{2m}{q^4} - \frac{2m}{5q^5} - \frac{m}{2q^6} \\
    &\geq& 1-\frac{g}{q+1}+\frac{gm}{2q^2(q+1)} -
        \frac{3m}{2q^2} - \frac{m}{3q^3} - \frac{2m}{q^4} - \frac{2m}{5q^5} - \frac{m}{2q^6} \\
    &\geq& 1- \frac{q-1}{3(q+1)} + \frac{(q-1)m}{6q^2(q+1)} -
           \frac{3m}{2q^2} - \frac{m}{3q^3} - \frac{2m}{q^4} - \frac{2m}{5q^5} - \frac{m}{2q^6} \\
    &=& \frac{2}{3} + \frac{2}{3(q+1)} - \frac{m}{q^2}
        \left(-\frac{q-1}{6(q+1)} + \frac{3}{2} + \frac{1}{3q} + \frac{2}{q^2}
        + \frac{2}{5q^3}+\frac{1}{2q^4}\right) \\
    &\geq &\frac{2}{3} + \frac{2}{3(q+1)} - \frac{1.45\cdot m}{q^2} 
  \end{eqnarray*}
\qed\end{proof}

\begin{corollary}\label{coro:cn-bound-2}
  Let $\F_q$ be of characteristic $p$, $n=p^{\ell}m$, with $\ell\geq 1$,
  $(m,p)=1$. Assume that $m< 2q^2$ is odd, $p\geq 3$ and $q\geq 9$. Then
    \[
      \CN_{q}(n) \geq q^n \left(\frac{2}{q+1} - \frac{2.735\cdot m}{q^2(q+1)} \right).
    \]
\end{corollary}
\begin{proof}
  The proof is very similar to that of Corollary~\ref{coro:cn-bound-1}.
  We compute
  \begin{align*}
     & 1-\frac{g}{q+1}+\frac{gm}{2q^2(q+1)} -
        \frac{m}{2q^2} - \frac{m}{q^p} - \frac{m}{3q^3} - \frac{2m}{q^4} - \frac{2m}{5q^5} - \frac{m}{2q^6} \\
    \geq& 1-\frac{q-1}{q+1}+\frac{(q-1)m}{2q^2(q+1)} -
           \frac{m}{2q^2} - \frac{4m}{3q^3} - \frac{2m}{q^4} - \frac{2m}{5q^5} - \frac{m}{2q^6} \\
    =& \frac{2}{q+1} - \frac{m}{q^2(q+1)}
        \left(-\frac{q-1}{2}+\frac{q+1}{2}+\frac{4(q+1)}{3q}+\frac{2(q+1)}{q^2}+\frac{2(q+1)}{5q^3}+\frac{q+1}{2q^4} \right)\\
    \geq& \frac{2}{q+1} - \frac{2.735\cdot m}{q^2(q+1)} .
  \end{align*}
\qed\end{proof}
The next corollary follows similarly from Proposition~\ref{prop:cn-bound-2}.
\begin{corollary} \label{coro:cn-bound-0}
  Let $\F_q$ be of characteristic $p$, and $1\leq n< 2q^2$, $(n,p)=1$.
  \begin{enumerate}
  \item For $n$ even, $q\geq 9$ and $q-1 \nmid n$,
    \[
      \CN_q(n) \geq q^n \left(\frac{1}{2} + \frac{1}{q+1} - \frac{0.86\cdot n}{q^2} \right) .
    \]
  \item For $n$ odd, $q\geq 8$
    \[
      \CN_q(n) \geq q^n \left(\frac{2}{q+1}-\frac{1.45\cdot n}{q^2(q+1)} \right).
    \]
  \end{enumerate}
\end{corollary}
\begin{proof}
For the first item, in Proposition~\ref{prop:cn-bound-2}, we assume that $g\leq (q-1)/2$, hence
\begin{align*}
\CN_q(n) & \geq q^n \left( 1 - \frac{q-1}{2(q+1)} + \frac{(q-1)n}{4q^2(q_1)} - \frac{n}{q^2} - \frac{n}{2q^3} \right) \\
 & = q^n \left( \frac 12 + \frac{1}{q+1} - \frac{n}{q^2} \left( \frac 34 + \frac{1}{2(q+1)} + \frac{1}{2q} \right) \right) \\
 & \geq q^n \left(\frac{1}{2} + \frac{1}{q+1} - \frac{0.86\cdot n}{q^2} \right) .
\end{align*}
\qed\end{proof}
\begin{theorem}\label{thm:cn-bound-final}
  Let $\F_q$ be of characteristic $p$, $n\in\N$ odd and $q\geq 8$. Then
  \begin{enumerate}
  \item For $n$ odd, $q\geq 8$,
    \begin{equation}
      \label{eq:cn-bound-odd}
      \CN_q(n) \geq q^n \left(\frac{2}{q+1}-\frac{1.45\cdot n}{q^2(q+1)} \right).
    \end{equation}
  \item For $n$ even, $q\geq 9$, $q-1\nmid n$,
    \begin{equation}
      \label{eq:cn-bound-even}
      \CN_q(n)\geq q^n\left( \frac{1}{2} + \frac{1}{q+1} - \frac{0.86\cdot n}{q^2} \right).
    \end{equation}
  \end{enumerate}
\end{theorem}
\begin{proof}
We first note that we may assume that $n< 2q^2$, since otherwise the bounds hold trivially.
  For the bound in Eq.~\eqref{eq:cn-bound-odd}, let $n=p^{\ell}m$, with $(m,p)=1$. First we consider the case $\ell\geq 1$. In
  this case, $p\geq 3$, $n\geq 3m$ and Corollary~\ref{coro:cn-bound-2} implies that
  \[
    \CN_q(n)\geq q^n\left(\frac{2}{q+1}-\frac{0.92\cdot n}{q^2(q+1)} \right).
  \]
  Suppose now that $\ell=0$, so that $(n,p)=1$. In this case, the stated bound
  is that of Corollary~\ref{coro:cn-bound-0}.
  
  For the bound in Eq.~\eqref{eq:cn-bound-even}, we first consider the case $p=2$, that is $n=2^\ell m$, $\ell\geq 1$, $m$ odd. In this case, from Corollary~\ref{coro:cn-bound-1},
  \[
  \CN_q(n) \geq q^n \left( \frac{2}{3} + \frac{2}{3(q+1)} - \frac{1.45\cdot m}{q^2} \right) \geq
  q^n \left( \frac{2}{3} + \frac{2}{3(q+1)} - \frac{0.725\cdot n}{q^2} \right) .
  \]
  Next, we consider the case $p\geq 3$ and $n=p^\ell m$, with $\ell\geq 1$ and $m$ even. From Corollary~\ref{coro:cn-bound-1},
  \[
  \CN_q(n) \geq q^n \left( \frac 12 + \frac{1}{q+1} - \frac{0.96\cdot m}{q^2} \right) \geq q^n \left( \frac 12 + \frac{1}{q+1} - \frac{0.32\cdot n}{q^2} \right) .
  \]
  Finally, we consider the case $p\geq 3$, $n$ even and $(n,p)=1$. From Corollary~\ref{coro:cn-bound-0}, we have
  \[
  \CN_q(n) \geq q^n \left(\frac{1}{2} + \frac{1}{q+1} - \frac{0.86\cdot n}{q^2} \right) .
  \]
  One easily checks that among the last three bounds, the latter is the weakest. 
\qed\end{proof}
\section{Proof of Theorem~\ref{thm:our_asymptotic}}\label{sec:proof1}
From Theorem~\ref{thm:our2}, we get $\PCN_q(n)>0$ provided that
\begin{equation}\label{eq:ip_pcn1}
\CN_q(n) > q^{n/2} W(q') \prod_{i=1}^k W_{l_i}(F_{l_i}') \theta_{l_i}(F_{l_i}').
\end{equation}
Clearly, $\theta_{l_i}(F_{l_i}')<1$ for all $i$ and $W_{l_i}(F_{l_i}') \leq 2^{n/l_i}$, so we have that
\begin{equation}\label{eq:ip_pcn2}
\prod_{i=1}^k W_{l_i}(F_{l_i}') \theta_{l_i}(F_{l_i}') < 2^{\sum_{i=1}^k n/l_i} = 2^{t(n)-1},
\end{equation}
where $t$ stands for the sum-of-divisors function. We now consider the case $n$ odd.
From Theorem~\ref{thm:cn-bound-final}, Lemma~\ref{lemma:w(r)} and Eqs.~\eqref{eq:ip_pcn1} and \eqref{eq:ip_pcn2},
we obtain the sufficient condition
\[
q^{n/4} \left( \frac{2}{q+1}-\frac{1.45\cdot n}{q^2(q+1)} \right) \geq 4.9 \cdot 2^{t(n)-1} .
\]
By Robin's theorem \cite{robin84},
\[
t(n) \leq e^\gamma n\log\log n + \frac{0.6483n}{\log\log n} , \ \forall n\geq 3 ,
\]
where $\gamma$ is the Euler-Mascheroni constant, hence the latter becomes
\begin{equation}\label{eq:ip_even_1}
q^{n/4} \left(\frac{2}{q+1}-\frac{1.45\cdot n}{q^2(q+1)} \right) \geq 4.9 \cdot 2^{n\left( \log\log n\cdot e^{0.578}+ \frac{0.6483}{\log\log n}\right)-1}.
\end{equation}
Assuming $n\geq 285$, a simple calculation shows that
\[
  4.9 \cdot 2^{n\left( \log\log n\cdot e^{0.578}+ \frac{0.6483}{\log\log n}\right)-1}
  \leq 2.5\cdot 2^{2n\log\log n},
\]
and since $n\leq q^2$ we obtain the condition
\begin{equation}\label{eq:suff-cond}
  \frac{0.55 \cdot q^{n/4}}{q+1} \geq 2.5 \cdot 2^{2n\log\log n}.
\end{equation}
Since $q^2\geq n\geq 285$, we have $q\geq 16$, so that $q+1\leq 1.0625\cdot q$ and the condition becomes
\[
  q^{n/4-1} \geq 4.83\cdot 4^{n\log\log n}.
\]
Since $q\geq \sqrt{n}$, it suffices
\[
  \frac{n-4}{8} \log n \geq n\log\log n \cdot \log 4 + \log 4.83,
\]
which is true for $n$ large enough.

For $n$ even, a similar argument leads to the sufficient condition
  \[
    q^{n/4}\left(\frac{1}{2}+\frac{1}{q+1}-\frac{0.86\cdot n}{q^2}\right) \geq 2.5\cdot 4^{n \log\log n}.
  \]
  For $n\leq 0.43\cdot q^2$ we obtain the sufficient condition
  \[
    \frac{q^{n/4}}{q+1} \geq 2.5\cdot 4^{n\log\log n}
  \]
  which is actually weaker than Eq.\eqref{eq:suff-cond}, and holds for $n$ large enough. The proof of Theorem~\ref{thm:our_asymptotic} is now complete.
\section{Proof of Theorem~\ref{thm:our_effective}}\label{sec:proof2}
Before we move on to the proof, we note that, in addition to the special cases mentioned in \cite{hachenberger13}, the case when $\F_{q^n}$ is completely basic over $\F_q$ can be excluded from our calculations. Namely, $\F_{q^n}$ is \emph{completely basic over $\F_q$} if every normal element of $\F_{q^n}$ is also completely normal over $\F_q$ and it is clear that in that case, Theorem~\ref{thm:pnbt} implies Conjecture~\ref{conj:mm}. Furthermore, we can characterize such extensions using the following, see \cite[Theorem~5.4.18]{hachenberger13} and, for a  proof, see \cite[Section~15]{hachenberger97}.
\begin{theorem}[\cite{hachenberger97}, Section~15]\label{thm:cbe}
Let $q$ be a power of the prime $p$. $\F_{q^n}$ is completely basic over $\F_q$ if and only if for every prime divisor $r$ of $n$, $r \nmid \ord_{(n/r)'}(q)$, where $(n/r)'$ stands for the $p$-free part of $n/r$ and $\ord_{(n/r)'}(q)$ for the multiplicative order of $q$ modulo $(n/r)'$.
\end{theorem}
The proof of Theorem~\ref{thm:our_effective} relies on computations performed with \textsc{SageMath}. We describe our steps for each item separately.
  \paragraph{Case 1: $n$ odd} Following the same steps as those that led us to Eq.~\eqref{eq:ip_even_1}, except that we now choose $a=12$ for the constant of Lemma~\ref{lemma:w(r)}, we obtain the condition
  \[
  q^{5n/12} \left(\frac{2}{q+1}-\frac{1.45\cdot n}{q^2(q+1)} \right) \geq 1.06\cdot 10^{24} \cdot 2^{n\left( \log\log n\cdot e^{0.578}+ \frac{0.6483}{\log\log n}\right)-1}.
  \]
  First, notice that the LHS of the latter is an increasing function of $q$ in the interval $n^{3/4}<q<n$, so it suffices to check its validity for $q=n^{3/4}$. It follows that the case $n\geq 14561$ is settled.
  
  Then we replace the term $2^{n\left( \log\log n\cdot e^{0.578}+ \frac{0.6483}{\log\log n}\right)-1}$ by $2^{t(n)-1}$ and, as before, $q$ by $n^{3/4}$ and check the resulting inequality for every $n<14561$, where $t(n)$ is computed explicitly for every $n$. The resulting inequality holds for every $n$, with the exception of 51 odd integers, with $135$ being the largest among them and for those $n$, we list all possible pairs $(q,n)$, where $q$ is a prime power with $n^{3/4}<q<n$. This leads to a list of 590 possible exceptional pairs.
  
This list is immediately reduced to a list of 31 pairs, once $1.06\cdot 10^{24}$ is replaced by the exact value of $c_{q',12}$, as described in Lemma~\ref{lemma:w(r)}, while all, but the 7 pairs $(q,n)$
\[
(9, 21),\, (11, 21),\, (16, 21),\, (17, 21),\, (11, 27),\, (13, 27) \text{ and } (16, 27)
\]
correspond to completely basic extensions.
  
  Finally, all 7 pairs satisfy the condition
  \[ q^{n/2} \left(\frac{2}{q+1}-\frac{1.45\cdot n}{q^2(q+1)} \right) > W(q') \prod_{i=1}^k W_{l_i}(F_{l_i}') \theta_{l_i}(F_{l_i}')
  \]
  if we compute every appearing value explicitly.
\paragraph{Case 2: $n$ even} As in the previous case, we begin with the condition 
\[
  q^{5n/12} \left(\frac{1}{2}+\frac{1}{q+1}-\frac{0.86\cdot n}{q^2}\right) \geq 1.06\cdot 10^{24} \cdot 2^{n\left( \log\log n\cdot e^{0.578}+ \frac{0.6483}{\log\log n}\right)-1}.
\]
Again, we replace $q$ by $n^{4/5}$ and verify that the latter holds for $n\geq 5719$.

Then we replace the term $2^{n\left( \log\log n\cdot e^{0.578}+ \frac{0.6483}{\log\log n}\right)-1}$ by $2^{t(n)-1}$ and, as before, $q$ by $n^{4/5}$ and check the resulting inequality for every $n<5719$, where $t(n)$ is computed explicitly for every $n$. The resulting inequality holds for every $n$, with the exception of 114 even integers, with $1680$ being the largest among them and for those $n$, we list all possible pairs $(q,n)$, where $q$ is a prime power with $n^{4/5}<q<n$. This leads to a list of 3250 possible exceptional pairs.

This list is furtherly reduced to 536 pairs, once $1.06\cdot 10^{24}$ is replaced by the exact value of $c_{q',12}$, as described in Lemma~\ref{lemma:w(r)} and, consequently, to 441 pairs if we exclude the pairs that turn out to correspond to completely normal extensions.

Our next step is to check the condition
\begin{equation}\label{eq:fin}
  q^{n/2} \left(\frac{1}{2}+\frac{1}{q+1}-\frac{0.86\cdot n}{q^2}\right) \geq W(q') \cdot 2^{t(n)-n-1} W_{1}(F_{1}')\theta_{1}(F_{1}'),
\end{equation}
which, as in Eq.~\eqref{eq:ip_pcn2}, derives from the fact
\[
\prod_{i=1}^k W_{l_i}(F_{l_i}') \theta_{l_i}(F_{l_i}') < W_{1}(F_{1}') \theta_{1}(F_{1}')\cdot 2^{\sum_{i=2}^k n/l_i} = W_{1}(F_{1}') \theta_{1}(F_{1}')\cdot 2^{t(n)-n-1}.
\]
First, we use Lemma~\ref{lemma:w(r)} and Eq.~\eqref{eq:fin} yields the condition
\[
  q^{5n/12} \left(\frac{1}{2}+\frac{1}{q+1}-\frac{0.86\cdot n}{q^2}\right) \geq c_{q',12} \cdot 2^{t(n)-n-1} W_{1}(F_{1}')\theta_{1}(F_{1}'),
\]
By checking the last condition, the list of possible exceptions reduces further to 47 pairs. Then, we explicitly compute $\prod_{i=1}^k W_{l_i}(F_{l_i}') \theta_{l_i}(F_{l_i}')$ and use this number over the above estimation and, this way, we reduce the number of possible exception pairs even more, to 26.

In the mentioned list of 26 pairs $(q,n)$, one finds the 18 pairs that are listed in Table~\ref{tab:1}
that are the pairs that fail to satisfy Eq.~\eqref{eq:fin} even after $W(q')$ is computed explicitly.
\begin{table}[h]
  \begin{center}
  \begin{tabular}{|c||c|c|c|c|c|c|c|c|c|c|c|}
\hline  $q$ & 5 & 7    & 8  & 9  & 11    & 13    & 17       & 19    & 23    & 29 & 41 \\ \hline
  $n$ & 6 & 8,\,10 & 12 & 10 & 12,\,16 & 16,\,20 & 18,\,24,\,36 & 24,\,30 & 24,\,48 & 60 & 60 \\ \hline
  \end{tabular}
  \end{center}
  \caption{Pairs $(q,n)$ on which the non-sieving methods were inadequate.\label{tab:1}}
\end{table}
\subsection{The sieve}
To deal with the persistent pairs $(q,n)$ of Table~\ref{tab:1}, we employ the Cohen-Huczynska~\cite{cohenhuczynska03,cohenhuczynska10} sieving techniques. In principle, those pairs can be handled by brute force, i.e., by finding appropriate examples and in fact such examples are already known for most of those pairs \cite{morganmullen96}, while for the rest, namely $(23, 48)$, $(29, 60)$ and $(41, 60)$, a modern computer is able to find such examples within a few minutes. Nonetheless, disposing of such pairs in a theoretical way, such as sieving, is desirable and we choose this path in this work.
\begin{proposition}[Sieving inequality]\label{propo:siev1}
Let $\{ r_1, \ldots , r_t \}$ be some divisors of $r$, where $r\mid q'$, such that $(r_i,r_j)=r_0$ for all $i\neq j$ and $\lcm(r_1,\ldots ,r_t)=r$, then
\[
\CN_q^r(n) \geq \sum_{i=1}^t \CN_q^{r_i}(n) - (t-1)\CN_q^{r_0}(n) .
\]
\end{proposition}
\begin{proof}
We denote by $\Ss(l)$ the set of $l$-primitive completely normal elements of $\F_{q^n}$ over $\F_q$, where $l$ may be any $r_i$. The statement is obvious for $t=1$. For $t=2$, we get that $\Ss(r_1)\cup\Ss(r_2)\subseteq\Ss(r_0)$ and $\Ss(r_1)\cap\Ss(r_2)=\Ss(q')$. The result follows after considering the cardinalities of the above sets.

Next, suppose the desired result holds for some $t=m\geq 2$. For $t=m+1$, if we denote by $r'$ the least common multiple of $r_2,\ldots r_{t+1}$, we observe that $\{r_1,r'\}$ satisfy the conditions for $t=2$. The desired result follows from the induction hypothesis.
\qed\end{proof}
\begin{proposition}\label{propo:siev2}
Let $q$ be a prime power, $n\in\N$ and $\{ p_1 ,\ldots ,p_t\}$ a set of prime divisors of $q^n-1$ (this set may be empty, in which case $t=0$), such that $\delta:=1-\sum_{i=1}^t p_i^{-1} >0$. If
\[
\CN_q(n) \geq q^{n/2} W(q_0) W_{l_1}(F_{l_1}') \cdots W_{l_k}(F_{l_k}') \left(\frac{t-1}{\delta}+2\right) \theta(\bq),
\]
where $q_0:=q'/p_1\cdots p_t$, then $\PCN_q(n)>0$.
\end{proposition}
\begin{proof}
Under the assumptions of the statement, Proposition~\ref{propo:siev1} implies
\[
\PCN_q(n) \geq \sum_{i=1}^t \CN_q^{q_0p_i}(n) - (t-1)\CN_q^{q_0}(n) .
\]
Next, we use the notation of the proof of Theorem~\ref{thm:our2} and by taking into account the analysis performed in its proof, the latter gives
\begin{align*}
\PCN_q(n) & \geq \sum_{i=1}^t \theta(q_0)\theta(p_i)\theta(\bq) (S_1+S_{2,q_0p_i})-(t-1)\theta(q_0)\theta(\bq)(S_1+S_{2,q_0}) \\
 & = \theta(q_0)\theta(\bq) \left( \delta S_1 + \sum_{i=1}^t \theta(p_i)S_{2,q_0p_i} - (t-1)S_{2,q_0} \right) ,
\end{align*}
which in turn yields
\[
\frac{\PCN_q(n)}{\theta(q_0)\theta(\bq)} \geq \delta S_1 + q^{n/2}W_{l_1}(F_{l_1}') \cdots W_{l_k}(F_{l_k}') W(q_0)\left( 1 + \sum_{i=1}^t ( \theta(p_i) \frac{W(q_0p_i)}{W(q_0)} - 1 ) \right)
\]
and by considering the fact that $W(q_0p_i)/W(q_0)=2$, we get that
\[
\frac{\PCN_q(n)}{\theta(q_0)\theta(\bq)} \geq \delta S_1 + q^{n/2}W_{l_1}(F_{l_1}') \cdots W_{l_k}(F_{l_k}') W(q_0)(t-1+2\delta).
\]
The last inequality combined with the fact that $S_1 = \CN_q(n)/\theta(\bq)$ completes the proof.
\qed\end{proof}
The latter implies that one may replace the term $W(q')$ in Eq.~\eqref{eq:fin} by $W(q'/s_1\cdots s_k)\cdot\Delta$, where $s_1,\ldots,s_k$ are prime divisors of $q'$ such that $\delta := 1-\sum_{i=1}^k 1/s_i > 0$ and $\Delta := (k-1)/\delta + 2$, so one has to look for appropriate prime divisors of $q'$. We attempt to find such divisors for all the pairs of Table~\ref{tab:1}. It turns out that this is possible for the pairs $(q,n)$ that are listed in Table~\ref{tab:3}, along with the appropriate primes.

\begin{table}[h]
\begin{center}
\begin{tabular}{|c|c||p{0.6\textwidth}||c|} \hline
$q$ & $n$ & Sieving primes & \# \\ \hline
8 & 12 & 109, 73, 37, 19, 13 & 5 \\ \hline
11 & 16 & 6304673, 7321, 61, 17, 5 & 5 \\ \hline
13 & 16 & 407865361, 14281, 17 & 3 \\ \hline
13 & 20 & 30941, 2411, 641 & 3 \\ \hline
17 & 18 & 1270657, 5653, 1423, 307, 19 & 5 \\ \hline
19 & 30 & 2460181, 1081291, 2251, 911, 271, 211, 151, 127, 61, 31, 11, 7 & 12 \\ \hline
23 & 24 & 83575993, 139921, 7549, 937, 79, 53, 37, 13, 11, 7 & 10 \\ \hline
23 & 48 & 483563163219889, 83575993, 12682129, 623009, 139921, 7549, 3697, 937 & 8 \\ \hline
29 & 60 & 4140278225341, 517475046481, 470925821, 111855481, 732541, 120691, 22111, 1061, 541, 421, 401 & 11 \\ \hline
41 & 60 & 8179560752161, 22616035021, 103826101, 11228251, 4555261, 579281, 382021, 22381, 4111, 1993, 1723, 1621, 761 & 13 \\ \hline
\end{tabular}
\end{center}
\caption{Pairs $(q,n)$ from Table~\ref{tab:1} that admit sieving, along with their sieving primes.\label{tab:3}}
\end{table}

The remaining pairs (listed in Table~\ref{tab:2}) were not dealt with theoretically.
However, those pairs $(q,n)$ satisfy $q\leq 97$ and $q^n<10^{50}$, i.e. Morgan and Mullen \cite{morganmullen96} have identified examples of primitive elements of $\F_{q^n}$ that are completely normal over $\F_q$. The proof of Theorem~\ref{thm:our_effective} is now complete.

\begin{table}[h]
\begin{center}
\begin{tabular}{|c||c|c|c|c|c|c|c|}
\hline
$q$ & 5 & 7      & 9  & 11 & 17      & 19 \\ \hline
$n$ & 6 & 8,\,10 & 10 & 12 & 24,\,36 & 24 \\ \hline
\end{tabular}
\end{center}
\caption{Pairs $(q,n)$ that were not dealt with theoretically.\label{tab:2}}
\end{table}
\section{Conclusions}\label{sec:conclusions}
The aim of this work is to establish the existence of primitive and completely normal elements for a larger range for the parameters $q,n$. We prove new sharper bounds for the number of completely normal elements of a given extension and use it to establish the existence of primitive and completely normal elements, using the method laid out in \cite{garefalakiskapetanakis18}. Our results hold asymptotically for $n$ up to roughly $q^2$ with the additional assumption that $q-1\nmid n$ when $n$ is even in Theorem~\ref{thm:our_asymptotic}.
Our method can be used to obtain effective results, as shown in Theorem~\ref{thm:our_effective}.
We believe that the range in Theorem~\ref{thm:our_asymptotic} cannot be significantly improved, without improving Theorem~\ref{thm:our2}. However, one should be able to improve Theorem~\ref{thm:our_effective} at the expense of significantly heavier computations.
An interesting problem for further work could be to remove the condition $q-1\nmid n$ for $n$ even and, more generally, to establish the existence of primitive and completely normal elements for any $n$ and $q$ such that $q-1\mid n$.
\begin{acknowledgements}
We are grateful to the anonymous reviewers for their valuable comments. Theodoulos Garefalakis was supported by the University of Crete Research Grant No.~10316.
\end{acknowledgements}
\end{document}